\documentclass[11pt,a4paper,reqno]{amsart}

\usepackage{amsmath, amsthm, amsopn, amssymb}
\usepackage{enumerate, enumitem}
\usepackage{hyperref}

\usepackage{geometry}
\newtheorem{thm}{Theorem}
\newtheorem{fact}[thm]{Fact}
\newtheorem{lemma}[thm]{Lemma}

\newtheorem{claim}{Claim}
\numberwithin{claim}{thm}

\theoremstyle{definition}
\newtheorem{dfn}[thm]{Definition}

\newcommand{\bp}{\mathbf{p}}
\newcommand{\bpr}{\bar{\mathbf{p}}}
\newcommand{\bq}{\mathbf{q}}

\newcommand{\cA}{\mathcal{A}}
\newcommand{\cC}{\mathcal{C}}
\newcommand{\cD}{\mathcal{D}}
\newcommand{\Fq}{\mathbb{F}_q}
\newcommand{\cJ}{\mathcal{J}}
\newcommand{\cP}{\mathcal{P}}
\newcommand{\Int}{\mathbb{Z}}
\newcommand{\Nat}{\mathbb{N}}
\newcommand{\Reals}{\mathbb{R}}
\newcommand{\cX}{\mathcal{X}}

\newcommand{\Ex}{\mathbb{E}}

\newcommand{\mum}{\mu_-}
\newcommand{\mup}{\mu_+}

\renewcommand{\le}{\leqslant}
\renewcommand{\ge}{\geqslant}

\newcommand{\pl}{\preceq}
\newcommand{\pg}{\succeq}

\title{Subsets of posets minimising the number of chains}

\author{Wojciech Samotij}
\address{School of Mathematical Sciences, Tel Aviv University, Tel Aviv 6997801, Israel}
\email{samotij@post.tau.ac.il}

\thanks{Research supported in part by the Israel Science Foundation grant 1147/14.}

\begin{document}

\maketitle

\begin{abstract}
  A well-known theorem of Sperner describes the largest collections of subsets of an $n$-element set none of which contains another set from the collection. Generalising this result, Erd\H{o}s characterised the largest families of subsets of an $n$-element set that do not contain a chain of sets $A_1 \subset \dotsc \subset A_k$ of an arbitrary length $k$. The extremal families contain all subsets whose cardinalities belong to an interval of length $k-1$ centred at $n/2$. In a far-reaching extension of Sperner's theorem, Kleitman determined the smallest number of chains of length two that have to appear in a collection of a given number $a$ of subsets of an $n$-element set. For every $a$, this minimum is achieved by the collection comprising $a$ sets whose cardinalities are as close to $n/2+1/4$ as possible. We show that the same is true about chains of an arbitrary length $k$, for all $a$ and $n$, confirming the prediction Kleitman made fifty years ago. We also characterise all families of $a$ subsets with the smallest number of chains of length $k$ for all $a$ for which this smallest number is positive. Our argument is inspired by an elegant probabilistic lemma from a recent paper of Noel, Scott, and Sudakov, which in turn can be traced back to Lubell's proof of Sperner's theorem.
\end{abstract}

\section{Introduction}
\label{sec:introduction}

The classical theorem of Sperner~\cite{Sp28} describes the largest families of subsets of a finite set none of which contains another set from the family. Originally a result in extremal set theory, it reached a broader mathematical audience as the key lemma in Erd\H{o}s' beautiful solution~\cite{Er45} to the Littlewood--Offord problem~\cite{LiOf43}. In the same paper~\cite{Er45}, generalising the theorem of Sperner, Erd\H{o}s described the largest families of sets not containing a \emph{chain} of sets $A_1 \subset \dotsc \subset A_k$ of an arbitrary length $k$. Among families of subsets of a set of cardinality $n$, the largest such families are those containing all subsets whose sizes belong to an interval of length $k-1$ centred at $n/2$. In particular, the largest size of a family of subsets of an $n$-element set that does not contain a \emph{$k$-chain} (chain of length $k$) is equal to the sum of the $k-1$ largest binomial coefficients.

It is natural to ask how many $k$-chains must appear in a family of subsets of an $n$-element set that is larger than Erd\H{o}s' bound. This problem was first considered by Erd\H{o}s and Katona, who conjectured that a family with $\binom{n}{\lfloor n/2 \rfloor} + t$ sets must contain at least $t \cdot \lceil \frac{n+1}{2} \rceil$ chains of length $2$. This conjecture was confirmed by Kleitman~\cite{Kl68}, who in fact determined the smallest number of $2$-chains in a family of any given size. For every $a$ with $0 \le a \le 2^n$, one such extremal family consists of $a$ sets whose cardinalities are as close to $n/2$ as possible. Kleitman went on to conjecture that the same families minimise the number of $k$-chains for every $k$.

Made almost fifty years ago, Kleitman's conjecture had laid dormant until his result was rediscovered several years ago by Das, Gan, and Sudakov~\cite{DaGaSu15} and, independently, by Dove, Griggs, Kang, and Sereni~\cite{DoGrKaSe14}. Apart from confirming the conjecture for every $k$ and all $a$ belonging to a small range above the sum of the $k-1$ largest binomial coefficients, the two papers~\cite{DaGaSu15, DoGrKaSe14} aroused renewed interest in the problem, as witnessed by~\cite{BaPeWa, BaWa, NoScSu}. In particular, Balogh and Wagner~\cite{BaWa} proved the conjecture for all $k$ and $a \le (1-\varepsilon)2^n$, provided that $n$ is sufficiently large with respect to $k$ and $\varepsilon$.

The main result of this paper is a resolution of Kleitman's conjecture for all $k$ and $n$. Additionally, for each $a$ larger than the sum of the $k-1$ largest binomial coefficients, we characterise all families of $a$ sets that contain the smallest number of $k$-chains among all families of this size. In the case $k = 2$, such characterisation was obtained by Das, Gan, and Sudakov~\cite{DaGaSu15}. In order to state our result formally, we need one definition. We shall say that a family $\cA$ of subsets of an $n$-element set $\Omega$ is \emph{centred} if it satisfies the following:
\begin{enumerate}[label={(\roman*)}]
\item
  \label{item:centred-1}
  If $A \in \cA$ and $B \subseteq \Omega$ satisfy $\big| |B| - n/2 \big| < \big| |A| - n/2 \big|$, then $B \in \cA$.
\item
  \label{item:centred-2}
  If for some $i < n/2$, the family $\cA$ contains some but not all subsets of $\Omega$ of size $i$ and some but not all subsets of $\Omega$ of size $n-i$, then either there are no $A, B \in \cA$ with $|A| = i$ and $|B| = n-i$ such that $A \subset B$ or there are no $A, B \not\in \cA$ with $|A| = i$ and $|B| = n-i$ such that $A \subset B$.
\end{enumerate}
We now state the main result of this paper. The case $k=2$ of the theorem, which is a strengthening of Kleitman's result because it characterises all extremal families, was already proved by Das, Gan, and Sudakov~\cite{DaGaSu15}.

\begin{thm}
  \label{thm:Kleitman}
  For all positive integers $a$, $k$, and $n$, every centred family of $a$ subsets of an $n$-element set contains the smallest number of chains of length $k$ among all families of this size. Moreover, if $k \ge 2$ and a centred family of $a$ subsets contains at least one chain of length $k$, then every non-centred family of this size contains strictly more such chains.
\end{thm}

In fact, the argument we use to prove Theorem~\ref{thm:Kleitman} is sufficiently flexible to establish analogous statements in a wider class of partially ordered sets that includes the boolean lattice (the family of all subsets of a given finite set). We shall defer the precise technical statement, Theorem~\ref{thm:main}, to Section~\ref{sec:defin-prel} and state here only one additional corollary of it in the setting of finite vector spaces. Generalising the definition given above, we shall say that a family $\cA$ of subspaces of an $n$-dimensional vector space $\Omega$ is centred if $\cA$ satisfies~\ref{item:centred-1} and~\ref{item:centred-2} above with $\subseteq$ ($\subset$) denoting `is a (proper) subspace of' and $|A|$ denoting the dimension of $A$.

\begin{thm}
  \label{thm:vector-spaces}
  For all positive integers $a$ and $k$, every centred family of $a$ subspaces of a finite vector space contains the smallest number of chains of length $k$ among all families of this size. Moreover, if $k \ge 2$ and a centred family of $a$ subspaces contains at least one chain of length $k$, then every non-centred family of this size contains strictly more such chains.
\end{thm}

The posets of subspaces of finite vector spaces were first studied in this context by Noel, Scott, and Sudakov~\cite{NoScSu}. The first assertion of Theorem~\ref{thm:vector-spaces} with $k=2$ was proved earlier by Balogh, Pet\v{r}\'{\i}\v{c}kov\'a, and Wagner~\cite{BaPeWa}.

The remainder of this paper is organised as follows. In order to highlight the main ideas to the reader, we start with a detailed sketch of our proof of Kleitman's conjecture, the first assertion of Theorem~\ref{thm:Kleitman}, which we give in Section~\ref{sec:proof-sketch}. The sketch is a completely rigorous argument, but it uses two important technical lemmas that we prove only later in the paper, in a more general form. Our method is based on and inspired by an elegant probabilistic lemma due to Noel, Scott, and Sudakov~\cite[Lemma~2.2]{NoScSu}, which can be traced back to Lubell's proof~\cite{Lu66} of Sperner's theorem. Section~\ref{sec:defin-prel} reviews basic notions related to partially ordered sets in preparation for the statement of Theorem~\ref{thm:main}. It is concluded with a short derivation of both Theorems~\ref{thm:Kleitman} and~\ref{thm:vector-spaces} from Theorem~\ref{thm:main} -- we simply verify that the boolean lattice and the family of subspaces of a finite vector space (ordered by inclusion) belong to the class of partially ordered sets covered by Theorem~\ref{thm:main}. In Section~\ref{sec:key-lemmas-proof}, we state the aforementioned two important technical lemmas, as well as a discrete analogue of Jensen's inequality, and use them to derive Theorem~\ref{thm:main}. Finally, we prove the two lemmas in Section~\ref{sec:proofs}, thus completing the proofs of Theorems~\ref{thm:Kleitman}, \ref{thm:vector-spaces}, and~\ref{thm:main}.

\section{A sketch of the proof of Kleitman's conjecture}
\label{sec:proof-sketch}

Let $P$ be the family of all subset of $\{1, \dotsc, n\}$ ordered by inclusion. List all sets in $P$ as $A_1, \dotsc, A_{2^n}$ ordered by the difference of their cardinality from the number $n/2-1/4$. For example, if $n$ is even, then the above sequence first lists all sets of size $n/2$, then all sets of size $n/2-1$, then all sets of size $n/2+1$, etc. For $0 \le a \le 2^n$, let $\cA_a = \{A_1, \dotsc, A_a\}$ and note that $\cA_a$ is centred. Kleitman conjectured that for each $a$, $k$, and $n$, the family $\cA_a$ has the smallest number of $k$-chains among all $a$-element subsets of $P$. More formally, letting $c_k(\cA)$ denote the number of $k$-chains formed by the sets in $\cA \subseteq P$ and $m_k(a) = c_k(\cA_a)$, the conjecture states that $c_k(\cA) \ge m_k(|\cA|)$ for every $\cA \subseteq P$.

We first observe that the function $a \mapsto m_k(a)$ is convex. Indeed, its discrete derivative, the function $a \mapsto m_k(a) - m_k(a-1)$, is nondecreasing. (It takes $n+2-k$ different values.) The much less obvious property of $m_k$ that our proof requires is that
\begin{equation}
  \label{eq:crucial-ppty}
  m_k\left( \sum_{i \in I} \binom{n}{i} \right) \le c_k\left( \bigcup_{i \in I} \binom{[n]}{i} \right) \text{ for all $I \subseteq \{0, \dotsc, n\}$}.
\end{equation}
In other words, \eqref{eq:crucial-ppty} states that the union of any collection of rank levels of $P$ contains at least as many $k$-chains as the centred family $\cA_a$ of the same cardinality. In other words, inequality~\eqref{eq:crucial-ppty} states that Kleitman's conjecture holds for each $\cA$ that is a union of full rank levels of $P$. (That is, for $\cA$ that contain either all of or none of the sets of every given cardinality.) In order to prove~\eqref{eq:crucial-ppty}, we will generalise it by replacing $\bigcup_{i \in I} \binom{[n]}{i}$ with a random subfamily of $P$ that contains each set with probability depending only on its cardinality, independently of all sets of different cardinalities, and replacing $\sum_{i \in I} \binom{n}{i}$ with the expected size of this random family. Our proof of this generalisation is elementary, but not particularly exciting, and thus we defer it to the main body of the paper. Instead, we just show how these two properties of $m_k$, convexity and~\eqref{eq:crucial-ppty}, imply Kleitman's conjecture.

Define the following integer-valued function $f$ on subsets of $P$. For every $\cX \subseteq P$,
\[
  f(\cX) = m_k\left( \sum_{A \in \cX} \binom{n}{|A|} \right) - c_k\left( \bigcup_{A \in \cX} \binom{[n]}{|A|}\right).
\]
Observe that~\eqref{eq:crucial-ppty} is equivalent to $f(\cC)$ being nonpositive for every chain $\cC$ of $P$. Now, fix an arbitrary $\cA \subseteq P$ and let $\cC$ be a uniformly selected random chain of (maximum) length $n+1$. Since $\cA \cap \cC$ is also a chain, we have $f(\cA \cap \cC) \le 0$. On the other hand, as for each $I \subseteq \{0, \dotsc, n\}$,
\[
  c_k\left( \bigcup_{i \in I} \binom{[n]}{i} \right) = \sum_{\substack{i_1, \dotsc, i_k \in I \\ i_1 < \dotsc < i_k}} c_k\left( \bigcup_{j = 1}^k \binom{[n]}{i_j} \right),
\]
then
\[
  0 \ge \Ex[f(\cA \cap \cC)] = \Ex\left[m_k\left(\sum_{A \in \cA \cap \cC} \binom{n}{|A|} \right)\right] - \Ex\left[ \sum_{\substack{A_1, \dotsc, A_k \in \cA \cap \cC \\ A_1 \subset \dotsc \subset A_k}} c_k\left(\bigcup_{j=1}^k \binom{[n]}{|A_j|} \right) \right].
\]
The random chain $\cC$ contains exactly one set with cardinality $i$ for each $i \in \{0, \dotsc, n\}$. By symmetry, each such set is equally likely to appear in $\cC$. It follows that $\Pr(A \in \cC) \cdot \binom{n}{|A|} = 1$ for each $A \in P$. Since $m_k$ is convex, then Jensen's inequality gives
\[
  \Ex\left[m_k\left(\sum_{A \in \cA \cap \cC} \binom{n}{|A|} \right)\right] \ge m_k\left(   \Ex\left[\sum_{A \in \cA \cap \cC} \binom{n}{|A|}\right]  \right) = m_k\left( \sum_{A \in \cA} \Pr(A \in \cC) \cdot \binom{n}{|A|}\right) = m_k(|\cA|),
\]
Generalising the above, if $0 \le i_1 < \dotsc < i_k \le n$, then the random chain $\cC$ contains exactly one $k$-chain of sets with cardinalities $i_1, \dotsc, i_k$. Moreover, by symmetry, each such chain is equally likely to appear in $\cC$. It follows that $\Pr(\{A_1, \dotsc, A_k\} \subseteq \cC) \cdot c_k\left(\bigcup_{j=1}^k \binom{[n]}{|A_j|} \right) = 1$ whenever $A_1 \subset \dotsc \subset A_k$. Consequently,
\[
  \Ex\left[ \sum_{\substack{A_1, \dotsc, A_k \in \cA \cap \cC \\ A_1 \subset \dotsc \subset A_k}} c_k\left(\bigcup_{j=1}^k \binom{[n]}{|A_j|} \right) \right] = \sum_{\substack{A_1, \dotsc, A_k \in A \\ A_1 \subset \dotsc \subset A_k}} \Pr(\{A_1, \dotsc, A_k\} \subseteq \cC) \cdot c_k\left(\bigcup_{j=1}^k \binom{[n]}{|A_j|} \right) = c_k(\cA).
\]
We conclude that $c_k(\cA) \ge m_k(|\cA|)$.

\section{Definitions and the main result}
\label{sec:defin-prel}

\subsection{Posets}

Let $P$ be a finite poset, that is, a finite set equipped with a partial order $\pl$. A set $L$ of elements of $P$ is called a \emph{chain} if the elements of $L$ are pairwise comparable. The \emph{length} of a chain is the number of its elements. For the sake of brevity, we shall refer to chains of length $k$ as \emph{$k$-chains}. The \emph{height} of $P$, which we shall denote by $h(P)$, is the largest length of a chain in $P$. We shall say that $P$ is \emph{homogeneous} if for every two maximal chains $L$ and $L'$ in $P$, there exists an automorphism of $P$ that maps $L$ to $L'$. (In particular, all maximal chains in a homogenous poset have the same length, and therefore every homogeneous finite poset is \emph{graded}.) We shall say that $P$ is \emph{symmetric} if it is isomorphic to the reverse of $P$, that is, if $(P, \pl)$ is order-isomorphic to $(P, \pg)$.

The \emph{ranking function} $r \colon P \to \Nat$ maps each $x \in P$ to the length of the longest chain in $P$ whose all elements are strictly smaller than $x$. This way, the elements of $P$ with zero rank are precisely the minimal elements of $P$ and the largest rank of an element of $P$ is $h(P) - 1$. It will be convenient to denote by $P_i$ the set of all elements of $P$ with rank $i$, which we refer to as the $i$th \emph{rank level} (or simply \emph{rank}) of $P$, that is,
\[
  P_i = r^{-1}(i) = \{x \in P \colon r(x) = i\}.
\]
The following definition generalises the two notions of centred families defined in the introduction.

\begin{dfn}
  \label{dfn:centred}
  A set $A$ of elements of a finite poset $P$ of height $n+1$ with ranking function $r \colon P \to \Nat$ is \emph{centred} if it satisfies the following:
  \begin{enumerate}[label={(\roman*)}]
  \item
    \label{item:centred-1}
    If $x \in A$ and $y \in P$ satisfy $|r(y)  - n/2| < |r(x) - n/2|$, then $y \in A$.
  \item
    \label{item:centred-2}
    If for some $i < n/2$, the set $A$ contains some but not all elements of $P_i$ and some but not all elements of $P_{n-i}$, then either there are no $x, y \in A$ with $r(x) = i$ and $r(y) = n-i$ such that $x \prec y$ or there are no $x, y \not\in A$ with $r(x) = i$ $x \prec y$ for all such $x, y \in A$.
  \end{enumerate}
\end{dfn}

For a positive integer $k$ and a set $A$ of elements of $P$, denote by $C_k(A)$ the set of all chains of length $k$ and let $c_k(A)$ denote the cardinality of this set. For a chain $L \subseteq P$ of length at most $k$, and a set $J \subseteq \Nat$, let $C_k'(L, J)$ be the set of all $k$-chains in $P$ that contain $L$ and whose remaining elements have ranks belonging to the set $J$, that is,
\[
  C_k'(L, J) = \{M \in C_k(P) \colon L \subseteq M \text{ and } r(M \setminus L) \subseteq J\}.
\]
Moreover, denote by $c_k'(L, J)$ the cardinality of the above set.

\begin{fact}
  \label{fact:ckLI-hom}
  Let $k$ be a positive integer, let $P$ be a homogeneous poset with ranking function $r \colon P \to \Nat$, and let $J \subseteq \Nat$. If chains $L_1$ and $L_2$ satisfy $r(L_1) = r(L_2)$, then
  \[
    c_k'(L_1, J) = c_k'(L_2, J).
  \]
\end{fact}
\begin{proof}
  Since $P$ is homogeneous and $r(L_1) = r(L_2)$, there is an automorphism $\varphi$ of $P$ that maps $L_1$ to $L_2$. To see this, consider arbitrary maximal chains $M_1 \supseteq L_1$ and $M_2 \supseteq L_2$ and observe that every automorphism of $P$ must preserve the rank. It is easy to see that $\varphi$ is a bijection between $C_k'(L_1, J)$ and $C_k'(L_2, J)$.
\end{proof}

Since in homogeneous posets, for every $k$ and every $J \subseteq \Nat$, the function $L \mapsto c_k'(L, I)$ depends only on $r(L)$, abusing the notation slightly, given an $I \subseteq \Nat$, we shall write $c_k'(I, J)$ to denote $c_k'(L, J)$ for some (every) chain $L$ with $r(L) = I$. (For the sake of brevity, we shall from now on often write $x$ in place of $\{x\}$ to denote the one-element set containing $x$.)

\begin{dfn}
  A homogenous poset $P$ with height $n+1$ is \emph{descending} if for every $0 < i < j \le n$,
  \[
    c_2'(i, j) \le c_2'(i-1, j-1).
  \]
  We shall say that $P$ is \emph{strictly descending} if the above inequality is strict for all such $i$ and $j$.
\end{dfn}

\subsection{The main result}
\label{sec:main-result}

With all definitions in place, we are finally ready to state the main technical result of this paper.

\begin{thm}
  \label{thm:main}
  Suppose that $P$ is a descending, homogeneous, and symmetric finite poset. For all positive integers $a$ and $k$, every centred set of $a$ elements of $P$ has the smallest number of $k$-chains among subsets of $P$ of this size. Moreover, if $k \ge 2$, $P$ is strictly descending, and a centred family of $a$ subsets contains at least one $k$-chain, then every non-centred family of this size contains strictly more $k$-chains.
\end{thm}

In order to deduce Theorems~\ref{thm:Kleitman} and~\ref{thm:vector-spaces}, it now suffices to check here that the boolean lattice and the poset of subspaces of a finite vector space are strictly descending, homogeneous, and symmetric. The ranking function of the boolean lattice is simply the cardinality. The homogeneity and symmetry properties are straightforward to verify. This poset is strictly descending as if $0< i < j \le n$, then
\[
  c_2'(i, j) = \binom{n-i}{j-i} < \binom{n-i+1}{j-i} = c_2'(i-1, j-1).
\]
The ranking function of the poset of subspaces of a finite vector space is simply the dimension. To see that it is homogeneous, take any two chains of maximal length. Each of them gives rise to a sequence of vectors that form a basis of the entire vector space. The linear operator that maps one of these sequences to the other induces an isomorphism of the entire vector space that maps one of the chains to the other. The function mapping each subspace of the vector space to its orthogonal complement is an isomorphism between the poset and its reverse; this proves that this poset is symmetric. To see that this poset is also strictly descending, note first that every finite vector space is isomorphic to $\Fq^n$ for some integer $n$ and prime power $q$. By elementary linear algebra, the number of subspaces of $\Fq^n$ of dimension $i$ is
\[
  \prod_{\ell = 0}^{i-1} \frac{q^n-q^\ell}{q^i-q^\ell}
\]
and if $0\le i < j \le n$, then $c_2(i,j)$ equals the number of subspaces of dimension $j-i$ of the space $\Fq^{n-i}$. Therefore, if $0 < i < j \le n$, then
\[
  c_2(i,j) = \prod_{\ell = 0}^{j-i-1} \frac{q^{n-i}-q^\ell}{q^{j-i}-q^\ell} < \prod_{\ell = 0}^{j-i-1} \frac{q^{n+1-i}-q^\ell}{q^{j-i}-q^\ell} = c_2(i-1, j-1).
\]
In particular, the poset of subspaces of every finite vector space is strictly descending.

\section{Lemmata and the derivation of Theorem~\ref{thm:main}}
\label{sec:key-lemmas-proof}

\subsection{Key lemmata}
\label{sec:key-lemmas}

The proof of Theorem~\ref{thm:main} will employ a discrete version of Jensen's inequality, which we present below. Suppose that $a$ and $b$ are integers with $a < b$ and consider a function $f \colon \{a, \dotsc, b\} \to \Reals$. The \emph{(backward) discrete derivative} of $f$ is the function $\Delta f \colon \{a+1, \dotsc, b\}$ defined by
\[
  \Delta f(x) = f(x) - f(x-1).
\]
We shall say that $f$ is \emph{convex} if $\Delta f$ is nondecreasing. The following statement is a discrete version of Jensen's inequality.

\begin{lemma}
  \label{lemma:Jensen}
  Let $a$ and $b$ be integers satisfying $a < b$ and suppose that $f \colon \{a, \dotsc, b\} \to \Reals$ is convex. If $X$ is an random variable taking values in $\{a, \dotsc, b\}$ such that $\Ex[X]$ is an integer, then $\Ex[f(X)] \ge f(\Ex[X])$. Moreover, this inequality is strict unless there are integers $c$ and $d$ such that $\Pr(c \le X \le d) = 1$ and $\Delta f$ is constant on $\{c+1, \dotsc, d\}$.
\end{lemma}

One may easily prove Lemma~\ref{lemma:Jensen} by invoking the standard version of Jensen's inequality for any convex function $g \colon [a,b] \to \Reals$ that extends $f$. One such function $g$ is defined by letting
\[
  g(x) = \left(\lceil x \rceil - x\right) \cdot f\left(\lfloor x \rfloor\right) + \left(x - \lfloor x \rfloor\right) \cdot f\left(\lceil x \rceil\right) \quad \text{for $x \in [a,b] \setminus \Int$}.
\]

We shall say that a set $I \subseteq \{0, \dotsc, n\}$ is \emph{centred} if $I$ is centred in the sense of Definition~\ref{dfn:centred} with $P$ being the standard ordering $\le$ of the set $\{0, \dotsc, n\}$. The following easy fact will be needed to derive the second part of Theorem~\ref{thm:main}, that is, the characterisation of subsets that minimise the number of chains.

\begin{fact}
  \label{fact:descending-unimodal}
  Suppose that $P$ is a homogeneous, symmetric, and descending poset of height $n+1$. Then $P$ is rank-unimodal. In particular, if $i, j \in [0,n]$ satisfy $|i-n/2| \le |j - n/2|$, then $|P_i| \ge |P_j|$. Moreover, if $P$ is strictly descending and $|i - n/2| < |j - n/2|$, then $|P_i| < |P_j|$.
\end{fact}

Suppose that $P$ is a homogeneous, symmetric, and descending poset of height $n+1$. It follows from Fact~\ref{fact:descending-unimodal} that for every $\ell \in \{1, \dotsc, n+1\}$ sets of $\ell$ largest ranks of $P$ are formed by ranks whose indices fall into (one of at most two) centred sets $I \subseteq \{0, \dotsc, n\}$ of cardinality $\ell$. Moreover, if $P$ is strictly descending, then the centred $\ell$-element sets are the only sets of $\ell$ largest ranks. Keeping this in mind, for every $\ell$ as above, let $I_\ell$ be one of the at most two centred subsets of $\{0, \dotsc, n\}$ with $\ell$ elements (the other centred set of the same size is $n - I_\ell$) and define
\[
  a_\ell = \sum_{i \in I_\ell} |P_i|.
\]
In other words, $a_\ell$ is the number of elements in the $\ell$ largest ranks of $P$. (If $P$ is the boolean lattice, then $a_\ell$ is the sum of the $\ell$ largest binomial coefficients.) For every integer $a$ with $0 \le a \le |P|$, let $m_k(a)$ denote the number of $k$-chains in some centred $a$-element subset of $P$. It is not very difficult to see that this definition does not depend on a particular choice of the $a$-element subset, as long as $P$ is homogeneous and symmetric. Our first key lemma states that the function $m_k$ is convex and that its discrete derivative $\Delta m_k$ increases at each $a_\ell$ with $\ell \ge k-1$.

\begin{lemma}
  \label{lemma:mk-convex}
  Suppose that $P$ is a homogeneous, symmetric, and descending finite poset. For each $k \ge 1$, the function $m_k \colon \{0, \dotsc, |P|\} \to \Nat$ is convex. Moreover, $\Delta m_k(a_\ell) < \Delta m_k(a_\ell+1)$ for every $\ell \ge k-1$.
\end{lemma}

Our second key lemma states that the union of any collection of rank levels of $P$ contains at least as many $k$-chains as a centred set of the same cardinality.

\begin{lemma}
  \label{lemma:main}
  Suppose that $P$ is a homogeneous, symmetric, and descending poset with height $n+1$. For every $I \subseteq \{0, \dotsc, n\}$ and every $k \ge 2$,
  \[
    c_k\left( \bigcup_{i \in I} P_i \right) \ge m_k\left( \sum_{i \in I} |P_i| \right).
  \]
  Moreover, if $P$ is strictly descending and $|I| \ge k$, then the above inequality is strict unless $I$ is centred.
\end{lemma}

We postpone the proofs of both lemmas to the next section and show first how to use them to prove Theorem~\ref{thm:main}. We only mention here that while proving Lemma~\ref{lemma:mk-convex} is rather easy and straightforward, establishing Lemma~\ref{lemma:main} requires work. The main idea behind our proof of the latter is to first generalise its statement by replacing $\bigcup_{i \in I} P_i$ with a random subset of $P$ that is uniform in each rank level and independent between different rank level. We then show that for every $a$, the expected number of $k$-chains in such a random set with expected cardinality $a$ is minimised when the probabilities assigned to the elements of $P$ are the largest for elements whose ranks are the closest to $n/2$; one such distribution is supported on centred families with $a$ elements. We achieve this by defining an operator on the set of all such distributions that pushes the support of a distribution towards the middle rank levels, at the same time decreasing the expected number of $k$-chains.

\subsection{Derivation of Theorem~\ref{thm:main}}
\label{sec:derivation-theorem}

The derivation of the first assertion of Theorem~\ref{thm:main} is simply a reiteration of the argument that we already presented in Section~\ref{sec:proof-sketch}. We mention here again that our argument draws inspiration from~\cite[Lemma~2.2]{NoScSu}. The second assertion of the theorem is obtained by the means of a careful analysis of the cases of equality in the two crucial inequalities appearing in the proof of the first assertion.

\begin{proof}[Proof of Theorem~\ref{thm:main}]
  Suppose that $P$ is a homogeneous, symmetric, and descending poset of height $n+1$. Define the following integer-valued function $f$ on subsets of $P$. For every $X \subseteq P$, let
  \[
    f(X) = m_k\left( \sum_{x \in X} |P_{r(x)}|\right) - c_k\left( \bigcup_{x \in X} P_{r(x)}\right).
  \]
  It follows from Lemma~\ref{lemma:main} that $f(C)$ is nonpositive for every chain $C$ of $P$. Indeed, every chain of $C$ contains at most one element of each rank. Now, fix an arbitrary $A \subseteq P$ and let $C$ be a uniformly selected random chain of maximum length. Since $A \cap C$ is also a chain, we have $f(A \cap C) \le 0$. On the other hand, as for each $I \subseteq \{0, \dotsc, n\}$,
  \[
    c_k\left( \bigcup_{i \in I} P_i \right) = \sum_{\substack{i_1, \dotsc, i_k \in I \\ i_1 < \dotsc < i_k}} c_k\left( \bigcup_{j = 1}^k P_{i_j} \right),
  \]
  the,
  \begin{equation}
    \label{eq:nonnegativity}
    0 \ge \Ex[f(A \cap C)] = \Ex\left[m_k\left(\sum_{x \in A \cap C} |P_{r(x)}| \right)\right] - \Ex\left[ \sum_{\substack{x_1, \dotsc, x_k \in A \cap C \\ x_1 \prec \dotsc \prec x_k}} c_k\left(\bigcup_{j=1}^k P_{r(x_j)} \right) \right].
  \end{equation}
  The random chain $C$ contains exactly one element of rank $i$ for each $i \in \{0, \dotsc, n\}$. Since $P$ is homogeneous, each such element is equally likely to appear in $C$. It follows that $\Pr(x \in C) \cdot |P_{r(x)}| = 1$ for each $x \in P$.  Since $m_k$ is convex, it follows from Jensen's inequality (Lemma~\ref{lemma:Jensen}) that
  \begin{equation}
    \label{eq:Jensen}
    \Ex\left[m_k\left(\sum_{x \in A \cap C} |P_{r(x)}| \right)\right] \ge m_k\left(   \Ex\left[\sum_{x \in A \cap C} |P_{r(x)}| \right]  \right) = m_k\left( \sum_{x \in A} \Pr(x \in C) \cdot |P_{r(x)}|\right) = m_k(|A|),
  \end{equation}
  Generalising the above, if $0 \le i_1 < \dotsc < i_k \le n$, then the random chain $C$ contains exactly one $k$-chain of elements with ranks $i_1, \dotsc, i_k$. Since $P$ is homogeneous, each such chain is equally likely to appear in $C$. It follows that $\Pr(\{x_1, \dotsc, x_k\} \subseteq C) \cdot c_k\left(\bigcup_{j=1}^k P_{r(x_j)} \right) = 1$ whenever $x_1 \prec \dotsc \prec x_k$. Consequently,
  \[
    \Ex\left[ \sum_{\substack{x_1, \dotsc, x_k \in A \cap C \\ x_1 \prec \dotsc \prec x_k}} c_k\left(\bigcup_{j=1}^k P_{r(x_j)} \right) \right] = \sum_{\substack{x_1, \dotsc, x_k \in A \\ x_1 \prec \dotsc \prec x_k}} \Pr(\{x_1, \dotsc, x_k\} \subseteq C) \cdot c_k\left(\bigcup_{j=1}^k P_{r(x_j)} \right) = c_k(A).
  \]
  We conclude that $c_k(A) \ge m_k(|A|)$.

  For the second assertion of the theorem, assume that $P$ is strictly descending, that $k \ge 2$, and that $A \subseteq P$ is an $a$-element set that satisfies $c_k(A) = m_k(a) > 0$. We need to show that $A$ is centred. Since $m_k(a) > 0 = m_k(a_{k-1})$, it must be that $a > a_{k-1}$. Crucially, since $c_k(A) = m_k(a)$, both~\eqref{eq:nonnegativity} and~\eqref{eq:Jensen} must hold with equality. The former implies that $f(A \cap C) = 0$ for every chain $C$ of maximum length, whereas the latter and Lemma~\ref{lemma:Jensen} imply that, as $C$ ranges over all chains of maximum length, the sum $\sum_{x \in A \cap C} |P_{r(x)}|$ takes values between some integers $c$ and $d$ for which $\Delta m_k(c+1) = \Delta m_k(d)$. Since the expected value of this sum is $a$, the second part of Lemma~\ref{lemma:mk-convex} implies that $a_\ell \le c \le a \le d \le a_{\ell+1}$ for some $\ell \ge k-1$. As $a_\ell$ is the sum of the $\ell$ largest ranks of $P$, then for every longest chain $C$, the set $A \cap C$ has at least $\ell$ elements. Moroever, $|A \cap C| > \ell$ unless $r(A \cap C)$ comprises some $\ell$ largest ranks of $P$, that is, unless $r(A \cap C)$ is centred, see the discussion following the statement of Fact~\ref{fact:descending-unimodal}. Furthermore, if $|A \cap C| > \ell \ge k-1$, then $r(A \cap C)$ is also centred, as $f(A \cap C) = 0$, by the second part of Lemma~\ref{lemma:main}. Consequently, for every chain $C$ of maximum length, the sum $\sum_{x \in A \cap C} |P_{r(x)}|$ is the sum of the $|A \cap C|$ largest ranks of $P$. Since this sum is bounded from above by $a_{\ell+1}$, then $|A \cap C| \le \ell+1$. To summarise, for each maximum chain $C$, the set $r(A \cap C)$ is centred and has either $\ell$ or $\ell+1$ elements. We claim that this property forces the set $A$ to be centred. Indeed, suppose first that $x \in A$ and $y \in P$ satisfy $|r(y) - n/2| < |r(x) - n/2|$. Since $x$ is contained in some longest chain, then $r(x)$ is among some $\ell+1$ largest ranks. As $r(y)$ is closer to $n/2$ than $r(x)$ is and $P$ is strictly descending, then $r(y)$ must belong to every set of largest $\ell$ ranks, see the discussion following the statement of Fact~\ref{fact:descending-unimodal}. As $y$ belongs to some longest chain $C$ and $r(A \cap C)$ contains some $\ell$ largest ranks of $P$, then $y \in A$. Finally, suppose that there is an $i < n/2$ such that $A$ contains some but not all elements of both $P_i$ and $P_{n-i}$. We claim that there cannot simultaneously be $x \in A \cap P_i$, $x' \in P_i \setminus A$, $y \in A \cap  P_{n-i}$, and $y' \in P_{n-i} \setminus A$ such that $x \prec y$ and $x' \prec y'$, which is precisely saying that $A$ satisfies~\ref{item:centred-2} in Definition~\ref{dfn:centred}. Suppose that there were such $x$, $x'$, $y$, and $y'$ and let $C$ and $C'$ be two chains of maximum length that contain $\{x,y\}$ and $\{x',y'\}$, respectively. Since both $r(A \cap C)$ and $r(A \cap C')$ are centred, then $r(A \cap C) \supseteq \{i, \dotsc, n-i\}$ and $r(A \cap C') \subseteq \{i+1, \dotsc, n-i-1\}$. But this contradicts the fact that $|r(A \cap C)| \le \ell+1 \le |r(A \cap C')| + 1$, completing the proof of the theorem.  
\end{proof}

\section{The proofs}
\label{sec:proofs}

In this section, we give proofs of the three key auxiliary results stated in Section~\ref{sec:key-lemmas}: Fact~\ref{fact:descending-unimodal} and Lemmas~\ref{lemma:mk-convex} and~\ref{lemma:main}. For the sake of brevity, given an integer $i$, we shall write $[i]$ to denote the set $\{1, \dotsc, i\}$. Moreover, for integers $i$ and $j$, we shall denote by $[i,j]$ the set of all integers $k$ with $i \le k \le j$, that is,
\[
  [i,j] = \{k \in \Int \colon i \le k \le j\}.
\]
We first give a proof of Fact~\ref{fact:descending-unimodal}, establishing two properties of the function $c_k'$ along the way.

\begin{fact}
  \label{fact:ckIJ-sym}
  Let $k$ be a positive integer, let $P$ be a homogeneous and symmetric poset of height $n+1$. For every $I, J \subseteq [0,n]$,
  \[
    c_k'(I, J) = c_k'(n-I, n-J).
  \]
\end{fact}
\begin{proof}
  Let $L$ be an arbitrary chain with $r(L) = I$ and let $\varphi$ be an arbitrary isomorphism between $(P, \pl)$ and $(P, \pg)$. Clearly, the same sets form chains in both $(P, \pl)$ and $(P, \pg)$. Moreover, since $P$ is homogeneous of height $n+1$, then for each $x \in P$, we have $r(\varphi(x)) = n-r(x)$. Thus, $\varphi$ is a bijection between $C_k'(L, J)$ and $C_k'(\varphi(L), n-J)$ and hence $c_k'(I,J) = c_k'(L, J) = c_k'(\varphi(L), n-J) = c_k'(n-I, n-J)$.
\end{proof}

\begin{fact}
  \label{fact:ckiJ-shift}
  Let $P$ be a descending homogeneous poset of height $n+1$. For every $k \ge 2$, every $i \in [n]$, every $1 \le s \le i$, and each $J \subseteq [i+1,n]$, we have
  \[
    c_k'(i, J) \le c_k'(i-s, J-s).
  \]
  Moreover, if $P$ is strictly descending and $|J| \ge k-1$, then the above inequality is strict.
\end{fact}
\begin{proof}
  It is clearly enough to prove the statement for $s = 1$. Moreover, we may assume that $|J| = k-1$, as
  \[
    c_k'(i, J) = \sum_{\substack{J' \subseteq J \\ |J'| = k-1}} c_k'(i, J').
  \]
  Suppose that $J = \{i + s_1, \dotsc, i + s_{k-1}\}$ for some $0 < s_1 < \dotsc < s_{k-1}$. Letting $s_0 = 0$, we may write
  \[
    c_k'(i, J) = \prod_{j = 1}^{k-1} c_2'(i+s_{j-1}, i+s_j) \le \prod_{j = 1}^{k-1} c_2'(i+s_{j-1}-1, i+s_j-1) = c_k'(i-1, J-1),
  \]
  where the above inequality is strict if $P$ is strictly descending.
\end{proof}

\begin{proof}[Proof of Fact~\ref{fact:descending-unimodal}]
  Since $P$ is homogeneous and symmetric, then $|P_i| = |P_{n-i}|$ for each $i \in [0,n]$. We may thus assume that $j < i \le n/2$. Counting comparable pairs $(x,y) \in P_i \times P_j$ in two different ways, we obtain
  \begin{equation}
    \label{eq:Pi-vs-Pj}
    |P_i| \cdot c_2'(i, j) = c_2(\emptyset, \{i,j\}) = |P_j| \cdot c_2'(j, i).
  \end{equation}
  By Facts~\ref{fact:ckIJ-sym} and~\ref{fact:ckiJ-shift} with $s = n - i - j$,
  \[
    c_2'(i, j) = c_2'(n-i, n-j) \le c_2'(j, i),
  \]
  which together with~\eqref{eq:Pi-vs-Pj} implies that $|P_i| \ge |P_j|$ and that this inequality is strict if $P$ is strictly descending.
\end{proof}

Let $P$ be a homogenous and symmetric poset of height $n+1$. Centred sets of elements of $P$ admit a somewhat more explicit description (as compared to Definition~\ref{dfn:centred}) using orderings $\mu \colon [n+1] \to [0, n]$ of the ranks $P_0, \dotsc, P_n$ of $P$ as $P_{\mu(1)}, \dotsc, P_{\mu(n+1)}$ that are nondecreasing in the distance from the middle, that is, such that $|\mu(\ell) - n/2|$ is nondecreasing in $\ell$. One can easily see that there are as many as $2^{\lceil n/2 \rceil}$ such orderings. We shall distinguish two of them, denoted $\mum$ and $\mup$, which we now define. The sequence $\mum$ is the ordering of the elements of $[0, n]$ by their growing distance from the number $n/2-1/4$. Similarly, $\mup$ is the ordering of the elements of $[0, n]$ by their growing distance from $n/2+1/4$. In other words, for each $\ell \in [n+1]$,
\[
  \mum(\ell) =
  \begin{cases}
    \left\lfloor\frac{n+(-1)^\ell \ell}{2}\right\rfloor & \text{if $n$ is odd,} \\
    \left\lceil\frac{n+(-1)^\ell \ell}{2}\right\rceil & \text{if $n$ is even,}
  \end{cases}
  \qquad
  \text{and}
  \qquad
  \mup(\ell) = n - \mum(\ell).
\]
Observe that for each $\ell \in [n+1]$, the sets $\mum([\ell])$ and $\mup([\ell])$ are the only centred $\ell$-element subsets of $[0, n]$. We may now explicitly describe a nested sequence of centred subsets of $P$. To this end, let $N = |P|$ and let $\mu$ be any ordering of the ranks of $P$ that is nondecreasing in the distance from $n/2$, e.g., $\mu \in \{\mum, \mup\}$. We order the elements of $P$ as $x_1, \dotsc x_N$ by exhaustively listing the elements of levels $P_{\mu(1)}, \dotsc, P_{\mu(n+1)}$, in this order, one-by-one. Formally, for every $a \in [N]$, we let $\ell$ be the largest integer such that $a > a_{\ell-1} = |P_{\mu(1)}| + \dotsc + |P_{\mu(\ell-1)}|$ and let $x_a$ be an arbitrary element of $P_{\mu(\ell)}$ that is not yet among $x_1, \dotsc, x_{a-1}$. For each $a \in [N]$, let $X_a = \{x_1, \dotsc, x_a\}$ and let $m_k(a) = c_k(X_a)$. It is not very difficult to see that $m_k$ does not depend on the choice of $\mu$ or on the particular ordering of the $x_a$ within the rank levels. Moreover, for each $\ell \in [n+1]$,
\[
  P_{\mu(\ell)} = \{x_a \colon a_{\ell-1} < a \le a_\ell \}.
\]
Let us point out that not all centred subsets of $P$ may be described in this fashion, but nonetheless each $a$-element centred $A \subseteq P$ does satisfy $c_k(A) = c_k(X_a)$.

\begin{proof}[{Proof of Lemma~\ref{lemma:mk-convex}}]
  Suppose that $P$ is a homogeneous, symmetric, and descending finite poset of height $n+1$. Our aim is to show that for each $a \in [1,|P|-1]$, the inequality
  \begin{equation}
    \label{eq:mk-convex}
    \Delta m_k(a) \le \Delta m_k(a+1)
  \end{equation}
  holds and that~\eqref{eq:mk-convex} is strict whenever $a = a_\ell$ for some $\ell \ge k$. The left-hand and the right-hand sides of~\eqref{eq:mk-convex} are $c_k'(x_a, r(X_a))$ and $c_k'(x_{a+1}, r(X_{a+1}))$, respectively. Since
  \[
    c_k'\big(x_{a+1}, r(X_{a+1})\big) = c_k'\big(x_{a+1}, r(X_{a+1} \setminus \{x_{a+1}\})\big) = c_k'\big(x_{a+1}, r(X_a)\big),
  \]
  then Fact~\ref{fact:ckLI-hom} implies that~\eqref{eq:mk-convex} holds with equality unless $x_a$ and $x_{a+1}$ have different ranks or, equivalently, unless $a = a_\ell$ for some $\ell$. In particular, we have $X_a = P_{\mu(1)} \cup \dotsc \cup P_{\mu(\ell)}$ and $x_{a+1} \in P_{\mu(\ell+1)}$. We split the remainder of the proof into two cases, depending on whether or not $\mu(\ell+1)$ is farther away from $n/2$ than $\mu(\ell)$.

  \noindent
  \textbf{Case 1. $|\mu(\ell+1) - n/2| = |\mu(\ell) - n/2|$.} \\
  In this case, $\mu(\ell+1) = n - \mu(\ell)$ and $\mu([\ell+1]) = n - \mu([\ell+1])$. Therefore by Facts~\ref{fact:ckLI-hom} and~\ref{fact:ckIJ-sym},
  \[
    \begin{split}
      c_k'\big(x_{a+1}, r(X_{a+1})\big) & = c_k'\big(\mu(\ell+1), \mu([\ell+1])\big) = c_k'\big(\mu(\ell), \mu([\ell+1])\big) \\
      & = c_k'\big(\mu(\ell), \mu([\ell])\big) + c_k'\big(\{\mu(\ell), \mu(\ell+1)\}, \mu([\ell-1])\big) \\
      & = c_k'\big(x_a, r(X_a)\big) + c_k'\big(\{\mu(\ell), \mu(\ell+1)\}, \mu([\ell-1])\big).
    \end{split}
  \]
  Thus, $c_k'\big(x_{a+1}, r(X_{a+1})\big) \ge c_k'\big(x_a, r(X_a)\big)$ and the inequality is strict when $c_k'\big(\{\mu(\ell), \mu(\ell+1)\}, \mu([\ell-1])\big)$ is nonzero, which happens if an only if $\ell+1 \ge k$.

  \noindent
  \textbf{Case 2. $|\mu(\ell+1) - n/2| > |\mu(\ell) - n/2|$.} \\
  In this case, $\mu([\ell]) = n - \mu([\ell])$ and $\mu(\ell+1) \in \{\min\mu([\ell]) - 1, \max\mu([\ell])+1\}$. In particular, letting $i = \min\mu([\ell])$, by Facts~\ref{fact:ckLI-hom} and~\ref{fact:ckIJ-sym},
  \[
    c_k'\big(x_{a+1}, r(X_{a+1})\big) = c_k'\big(\mu(\ell+1), \mu([\ell+1])\big) = c_k'\big(\mu(\ell+1), \mu([\ell])\big) = c_k'\big(i-1, \mu([\ell])\big)
  \]
  and
  \[
    c_k'\big(x_a, r(X_a)\big) = c_k'\big(\mu(\ell), \mu([\ell])\big) = c_k'\big(i, \mu([\ell])\big).
  \]
  Thus, it suffices to verify that $c_k'\big(i, \mu([\ell])\big) \le c_k'\big(i-1, \mu([\ell])\big)$ and that the inequality is strict when $\ell \ge k-1$. To this end, let $x$ and $y$ be two elements of $P$ satisfying $x \prec y$ with $r(x) = i-1$ and $r(y) = i$ and observe that for each $I \subseteq [i,n]$, there is a canonical injection from $C_k'(y, I)$ to $C_k'(x, I)$, defined by replacing $y$ by $x$ in each $k$-chain. Note moreover, that the image of $C_k'(y, \mu[\ell])$ via this injection cannot contain any $k$-chain passing through both $x$ and $y$ and that there is at least one such chain when $\ell \ge k-1$.
\end{proof}

\begin{proof}[{Proof of Lemma~\ref{lemma:main}}]
  Fix an $I \subseteq [0,n]$, let $\bq = (q_0, \dotsc, q_n) \in \{0, 1\}^{n+1}$ be its characteristic vector, that is, $q_i = 1$ if $i \in I$ and $q_i = 0$ otherwise, and let 
  \[
    a = \sum_{i \in I} |P_i| = \sum_{i=0}^n q_i|P_i|.
  \]
  Define
  \[
    \cD = \left\{(p_0, \dotsc, p_n) \in [0,1]^{n+1} \colon p_i |P_i| \in \Nat \text{ for all $i$ and } \sum_{i=0}^n p_i |P_i| = a \right\}
  \]
  and note that $\bq \in \cD$. Given a $\bp = (p_0, \dotsc, p_n) \in \cD$, let $R$ be the random set of elements of $P$ formed by including each $x \in P$ with probability $p_{r(x)}$, independently of other elements. Observe that the expected size of $R$ is precisely $a$ and define
  \begin{equation}
    \label{eq:wk}
    w_k(\bp) = \Ex[c_k(R)] = \sum_{|J| = k} c_k'(\emptyset, J) \cdot \prod_{j \in J} p_j.
  \end{equation}

  \begin{claim}
    \label{claim:wk-minimum}
    The function $w_k \colon \cD \to \Reals$ achieves its minimum at a vector $\bp$ for which there exist $\ell \in [n+1]$ and $\mu \in \{\mum, \mup\}$ such that $p_{\mu(m)} = 1$ for all $m < \ell$ and $p_{\mu(m)} = 0$ for all $m > \ell$. Moreover, if $P$ is strictly descending and $|I| \ge k$, then $w_k(\bq) > w_k(\bp)$ unless $I = \mum([|I|])$ or  $I = \mup([|I|])$, that is, unless $I$ is centred.
  \end{claim}

  Let us first argue that Claim~\ref{claim:wk-minimum} implies the assertion of the lemma. To this end, let $\bp \in \cD$ be a vector minimising $w_k$ on $\cD$ of the form described in the claim. Since
  \[
    c_k\left( \bigcup_{i \in I} P_i \right) = w_k(\bq) \ge w_k(\bp),
  \]
  and the above inequality is strict unless $|I| < k$, $I = \mum([|I|])$, or $I = \mup([|I|])$, it is clearly enough to show that $w_k(\bp) = m_k(a)$. Let $M = p_{\mu(\ell)} |P_{\mu(\ell)}|$ and let $R'$ be the union of $\bigcup_{m=0}^{\ell-1} P_{\mu(m)}$ and the uniformly chosen random $M$-element subset of $P_{\mu(\ell)}$. Since every chain contains at most one element of each rank,
  \[
    \Ex[c_k(R')] = \Ex[c_k(R)] = w_k(\bp).
  \]
  On the other hand, it is not hard to see that $c_k(R')$ is constant. Indeed, since $R' \setminus P_{\mu(\ell)}$ is a union of full ranks, Fact~\ref{fact:ckLI-hom} implies that
  \begin{multline}
    c_k(R') = c_k\big(R' \setminus P_{\mu(\ell)}\big) + \sum_{x \in R' \cap P_{\mu(\ell)}} c_k'\big(x, r(R' \setminus P_{\mu(j)})\big) = c_k'\big(\emptyset, \mu([\ell-1])\big) + |R' \cap P_{\mu(\ell)}| \cdot c_k'\big(\mu(\ell), \mu([\ell-1])\big) \\
    = c_k'\big(\emptyset, \mu([\ell-1])\big) + M \cdot c_k'\big(\mu(\ell), \mu([0, \ell-1])\big).
  \end{multline}
  Finally, it is clear that $\Pr(R' = X_a) = 1/\binom{|P_{\mu(j)}|}{M} > 0$ and thus
  \[
    w_k(\bp) = \Ex[c_k(R')] = c_k(X_a) = m_k(a).
  \]

  We now prove the claim. To this end, we shall define a compression operator $\Phi \colon \cD \to \cD$ that, roughly speaking, pushes the mass of every distribution $\bp$ towards the middle rank levels. We shall then verify that $\Phi$ has the following three properties:
  \begin{enumerate}[label={(P\arabic*)}]
  \item
    \label{item:varphi-monotone}
    For every $\bp \in \cD$, we have $w_k(\Phi(\bp)) \le w_k(\bp)$.
  \item
    \label{item:varphi-fixpoint}
    For every $\bp \in \cD$, there is a $K$ such that $\Phi^{(K)}(\bp)$ has the form described in Claim~\ref{claim:wk-minimum}.
  \item
    \label{item:varphi-stability}
    If $|I| \ge k$ but $I \neq \mum([|I|])$ and $I \neq \mup([|I|])$, then $w_k(\Phi(\bq)) < w_k(\bq)$.
  \end{enumerate}
  It is not hard to see that the existence of such a $\Phi$ would establish the claim. Indeed, since $\cD$ is finite, $w_k$ achieves its minimum on some $\bp' \in \cD$. By property~\ref{item:varphi-fixpoint}, applying $\Phi$ to $\bp'$ some number of times produces a vector $\bp$ of the form described in the statement of the claim. By property~\ref{item:varphi-monotone}, we have $w_k(\bp) \le w_k(\bp') = \min w_k(\cD)$. Finally, property~\ref{item:varphi-stability} implies the last assertion of the claim. It thus suffices to construct such a compression operator $\Phi$.

  Fix some $\bp \in \cD$, let $\bpr = (p_n, \dotsc, p_0)$ and note that $\bpr \in \cD$, as $P$ is symmetric. Fact~\ref{fact:ckIJ-sym} implies that
  \begin{equation}
    \label{eq:wk-symmetry}
    w_k(\bpr) = \sum_{|J| = k} c_k'(\emptyset, J) \cdot \prod_{j \in J} p_{n-j} = \sum_{|J| = k} c_k'(\emptyset ,n-J) \cdot \prod_{j \in J} p_{n-j} = w_k(\bp).    
  \end{equation}
  Furthermore, $\bp$ has the form described in Claim~\ref{claim:wk-minimum} if and only if $\bpr$ has such form, with $\mup$ playing the role of $\mum$ and vice-versa. Let $i$ be the smallest index such that $\max\{p_i, p_{n-i}\} > 0$. By~\eqref{eq:wk-symmetry}, we may assume that $p_i > 0$, as otherwise we replace $\bp$ with $\bpr$. Now, let $i'$ be the smallest index greater than $i$ such that $p_{i'} \neq 1$. (If such $i'$ does not exist, then $\bp = (1, \dotsc, 1)$ and $\bp$ is the only element of $\cD$.) Clearly, $i+1 \le i' \le n-i+1$. If $i' = n-i+1$, then $\bp$ has the required form, with $\mu = \mup$ and $\ell = (\mup)^{-1}(i)$, and we let $\Phi(\bp) = \bp$. Thus, we may assume that $i' \le n-i$. If $i' = n-i$ and $p_{i'} = 0$, then $\bp$ has the required form, with $\mu = \mum$ and $\ell = (\mum)^{-1}(i)$, and we again let $\Phi(\bp) = \bp$. Thus, we may assume that either $i' < n-i$ or that $i' = n-i$ but $0 < p_{i'} < 1$. By~\eqref{eq:wk-symmetry}, we may assume that in the latter case $p_i \le p_{n-i}$, as otherwise we may replace $\bp$ with $\bpr$. (Note also for future reference that the latter case is impossible if $\bp = \bq$, as all coordinates of $\bq$ are either $0$ or $1$.) Let
  \[
    \delta = \min\left\{ p_i, (1-p_{i'}) \cdot \frac{|P_{i'}|}{|P_i|}\right\} \qquad \text{and} \qquad \delta' = \delta \cdot \frac{|P_i|}{|P_{i'}|}.
  \]
  Define $\Phi(\bp) = (p_0', \dotsc, p_n')$ by
  \[
    p_j' =
    \begin{cases}
      p_i - \delta & \text{if $j = i$}, \\
      p_{i'} + \delta' & \text{if $j = i'$}, \\
      p_j & \text{if $j \not\in \{i, i'\}$},
    \end{cases}
  \]
  and observe that $\Phi(\bp) \in \cD$. Informally speaking, $\Phi$ moves as much mass as possible from $p_i$ to $p_{i'}$ subject to $p_i$ and $p_{i'}$ remaining in the interval $[0,1]$; in particular, either $p_i' = 0$ or $p_{i'}' = 1$. It is not difficult to see that for large enough $K$, the vector $\Phi^{(K)}(\bp)$ has the special form described in the claim. Indeed, if $i' = n-i$, then already $\Phi(\bp)$ has this form. Otherwise, letting $h \colon \cD \to \Nat$ be the function defined by 
  \[
    h(\bp) = \sum_{i=0}^n |2i - n| \cdot p_i |P_i|,
  \]
  one sees that $h(\Phi(\bp)) < h(\bp)$.

  We claim that $w_k(\Phi(\bp)) \le w_k(\bp)$ and that if $P$ is strictly descending and $|I| \ge k$, then this inequality is strict. To this end, let
  \begin{align*}
    \cJ_1 & = \{J \subseteq [0,n] \colon |J| = k-1 \text{ and } i, i' \not\in J\}, \\
    \cJ_2 & = \{J \subseteq [0,n] \colon |J| = k-2 \text{ and } i, i' \not\in J\}
  \end{align*}
  and observe that
  \begin{multline}
    \label{eq:wk-change}
    w_k(\bp) - w_k(\Phi(\bp)) = \sum_{J \in \cJ_1} c_k'(\emptyset, J \cup \{i\}) \cdot \delta \prod_{j \in J} p_j - \sum_{J \in \cJ_1} c_k'(\emptyset, J \cup \{i'\}) \cdot \delta' \prod_{j \in J} p_j \\
    + \sum_{J \in \cJ_2} c_k'(\emptyset, J \cup \{i, i'\}) \cdot (p_ip_{i'} - (p_i-\delta)(p_{i'} + \delta')) \prod_{j \in J} p_j.
  \end{multline}
  We first claim that each summand in the third sum in~\eqref{eq:wk-change} is nonnegative. Clearly, it suffices to show that
  \begin{equation}
    \label{eq:pipip-change}
    p_ip_{i'} - (p_i-\delta)(p_{i'} + \delta') \ge 0.
  \end{equation}
  By our choice of $\delta$ and $\delta'$, either $p_i - \delta = 0$ or $p_{i'} + \delta' = 1$. In the former case, inequality~\eqref{eq:pipip-change} holds trivially. Assume that the latter case holds. If $i' = n-i$, then $\delta = \delta'$, as $|P_i| = |P_{i'}|$ by Fact~\ref{fact:descending-unimodal}, and $p_i \le p_{i'}$ by the assumption we made above. Thus,
  \[
    p_ip_{i'} - (p_i - \delta)(p_{i'} + \delta') = \delta(p_{i'} - p_i) + \delta^2 \ge 0.
  \]
  Finally, if $i' < n-i$, then $|P_{i'}| \ge |P_i|$ by Fact~\ref{fact:descending-unimodal}. In particular, $\delta = (1-p_{i'}) |P_{i'}| / |P_i| \ge 1-p_{i'}$ and consequently,
  \[
    p_ip_{i'} - (p_i - \delta)(p_{i'} + \delta') = p_ip_{i'} - p_i + \delta \ge p_ip_{i'} - p_i + 1 - p_{i'} = (1-p_i)(1-p_{i'}) \ge 0.
  \]
  We now move to the heart of the matter, which is comparing the first two sums.

  Let $\sigma \colon [0,n] \to [0,n]$ be the function that reverses the interval $[i+1,i'-1]$, that is, let
  \[
    \sigma(j) =
    \begin{cases}
      i+i'-j & \text{if $i < j < i'$},\\
      j & \text{otherwise}.
    \end{cases}
  \]
  It is easy to see that $\sigma$ induces an involution in $\cJ_1$. Indeed, $\sigma$ is an involution in $\cP([0,n])$ and $\sigma(J) \in \cJ_1$ for each $J \in \cJ_1$.

  \begin{claim}
    \label{claim:delta-ck-comp}
    For each $J \in \cJ_1$ with $J \subseteq [i+1, n-i-1]$,
    \begin{equation}
      \label{eq:delta-ck-comp}
      \delta \cdot c_k'(\emptyset, J \cup \{i\}) \ge \delta' \cdot c_k'(\emptyset, \sigma(J) \cup \{i'\}).
    \end{equation}
    Moreover, if $P$ is strictly descending and $i' < n-i$, then the above inequality is strict.
  \end{claim}
  
  We first argue that Claim~\ref{claim:delta-ck-comp} implies that the right-hand side of~\eqref{eq:wk-change} is nonnegative. Indeed, if there was a $j' \in J$ with $j' < i$ or $j' > n-i$, then $\prod_{j \in J} p_j = \prod_{j \in \sigma(J)} p_j = 0$ as $p_{j'} = p_{\sigma(j')}= 0$ by our choice of $i$. Crucially, since $p_j = 1$ for all $j \in [i+1, i'-1]$, then for each $J$ as in the claim,
  \[
    \prod_{j \in J} p_j = \prod_{j \in J \setminus [i+1, i'-1]} p_j = \prod_{j \in \sigma(J \setminus [i+1, i'-1])} p_j = \prod_{j \in \sigma(J)} p_j.
  \]
  Therefore, summing~\eqref{eq:delta-ck-comp} over all $J$ implies that the first line of the right-hand side of~\eqref{eq:wk-change} is nonnegative, and we have already shown above that the second line is also nonnegative. Moreover, if $P$ is strictly descending, then inequality~\eqref{eq:delta-ck-comp} is strict. Since if $|I| \ge k$ and $\bp = \bq$, then $i' \neq n-i$ and there is a $J \in \cJ_1$ such that $\prod_{j \in J} p_j > 0$ and consequently, the right-hand side of~\eqref{eq:wk-change} is positive.

  It suffices to prove Claim~\ref{claim:delta-ck-comp}. To this end, fix a $J \in \cJ_1$ and suppose that
  \[
    J \cap [i+1, i'-1] = \{i + s_1, \dotsc, i+s_r\}
  \]
  for some $r \ge 0$ and $0 < s_1 < \dotsc < s_r < i'-i$. Since $i < \min J$, then
  \[
    c_k'(\emptyset, J \cup \{i\}) = |P_i| \cdot c_k'(i, J) = |P_i| \cdot c_{r+1}'(i, \{i+s_1, \dotsc, i+s_r\}) \cdot c_{k-r}(i+s_r, J \cap [i'+1, n]).
  \]
  Similarly, as $\sigma(J) \cap [0, i'-1] = \sigma(\{i+s_1, \dotsc, i+s_r\}) = \{i'-s_1, \dotsc, i'-s_r\}$, then
  \[
    c_k'(\emptyset, \sigma(J) \cup \{i'\}) = |P_{i'}| \cdot c_{r+1}'(i', \sigma(J)) = |P_{i'}| \cdot c_{r+1}(i', \{i'-s_1, \dotsc, i'-s_r\}) \cdot c_{k-r}'(i', J \cap [i'+1, n]).
  \]
  By Fact~\ref{fact:ckIJ-sym},
  \[
    c_{r+1}(i', \dotsc, \{i'-s_1, \dotsc, i'-s_r\}) = c_{r+1}(n-i', \{n-i'+s_1, \dotsc, n-i'+s_r\}).
  \]
  Finally, note three facts. First, $\delta|P_i| = \delta'|P_{i'}|$. Second, since $i + s_r < i'$, then
  \[
    c_{k-r}'(i+s_r, J \cap [i'+1, n]) \ge c_{k-r}'(i', J \cap [i'+1, n]).
  \]
  Indeed, letting $x \in P_{i+s_r}$ and $y \in P_{i'}$ be arbitrary elements with $x \prec y$, there is a natural injection from $C_{k-r}'(y, J \cap [i'+1, n])$ to $C_{k-r}'(x, J \cap [i'+1, n])$, simply replacing $y$ with $x$. Finally, as $i \le n-i'$, then by Fact~\ref{fact:ckiJ-shift}, we have
  \[
    c_{r+1}'(i, \{i+s_1, \dotsc, i+s_r\}) \ge c_{r+1}'(n-i', \{n-i'+s_1, \dotsc, n-i'+s_r\}).
  \]
  Moreover, if $P$ is strictly descending and $i' < n-i$, then the above inequality is strict. This proves Claim~\ref{claim:delta-ck-comp}, thus completing the proof of Lemma~\ref{lemma:main}.
\end{proof}

\bibliography{poset-supsat}
\bibliographystyle{amsplain}

\end{document}